\documentclass[letterpaper,11pt]{article}

\def\[#1\]{\begin{align}#1\end{align}}
\def\(#1\){\begin{align*}#1\end{align*}}

\def\argmax{\operatornamewithlimits{arg\,max}}
\def\argmin{\operatornamewithlimits{arg\,min}}

\newcommand{\bprf}{\begin{proof}}
\newcommand{\eprf}{\end{proof}}
\newcommand{\blem}{\begin{lemma}}
\newcommand{\elem}{\end{lemma}}

\newcommand{\bE}{\mathbb{E}}

\usepackage{url}
\usepackage{ifthen}
\usepackage{cite}

\usepackage{amssymb,amsfonts,amstext,amsthm,mathrsfs}
\usepackage{graphics,latexsym,epsfig,psfrag,wrapfig,comment,paralist}
\usepackage{xspace}
\usepackage{subcaption,bm,multirow}
\usepackage{booktabs}
\usepackage{mathdots}
\usepackage{bbold}

\usepackage{algorithm, algorithmicx, algpseudocode}

\usepackage[CJKbookmarks=true,
            bookmarksnumbered=true,
            bookmarksopen=true,
            colorlinks=true,
            citecolor=blue,
            linkcolor=blue,
            anchorcolor=red,
            urlcolor=blue
            ]{hyperref}
\usepackage{natbib}
\setcitestyle{authoryear,round}
\newtheoremstyle{break}
  {\topsep}{\topsep}%
  {\itshape}{}%
  {\bfseries}{}%
  {\newline}{}%
\usepackage{listings}
\usepackage{tikz}
\usetikzlibrary{shapes}
\usetikzlibrary{positioning, arrows, matrix, calc, patterns, fit}
\usepackage{caption}
\usepackage{array}
\usepackage{mdwmath}
\usepackage{multirow}
\usepackage{mdwtab}
\usepackage{eqparbox}
\usepackage{multicol}
\usepackage{amsfonts}
\usepackage{tikz}
\usepackage{multirow,bigstrut,threeparttable}
\usepackage{amsthm}
\usepackage{array}
\usepackage{bbm}
\usepackage{epstopdf}
\usepackage{mdwmath}
\usepackage{mdwtab}
\usepackage{eqparbox}
\usepackage{tikz}
\usepackage{latexsym}
\usepackage{amssymb}
\usepackage{bm}
\usepackage{graphicx}
\usepackage{mathrsfs}
\usepackage{epsfig}
\usepackage{psfrag}
\usepackage{setspace}

\usepackage[cmex10]{amsmath} 

\newcommand{\TV}{\mathsf{TV}}
\newcommand{\tTV}{\widetilde{\mathsf{TV}}}
\newcommand{\bR}{\mathbb{R}}

\newcommand{\GG}{\mathcal{G}}

\newtheorem{theorem}{Theorem}
\newtheorem{lemma}{\textbf{Lemma}}
\usepackage[normalem]{ulem}

\usepackage{color}
\begin{document}
\title{When does the Tukey Median work? } 


\author{Banghua Zhu, Jiantao Jiao, Jacob Steinhardt\thanks{Banghua Zhu is with the Department of Electrical Engineering and Computer Sciences, University of California, Berkeley. Jiantao Jiao is with the Department of Electrical Engineering and Computer Sciences and the Department of Statistics, University of California, Berkeley. Jacob Steinhardt is with the Department of Statistics and the Department of Electrical Engineering and Computer Sciences, University of California, Berkeley. Email: \{banghua, jiantao,jsteinhardt\}@berkeley.edu.}}
\date{\today}

\maketitle

\begin{abstract}
  We analyze the performance of the Tukey median estimator under total variation ($\TV$) distance corruptions. Previous results  show that under Huber's additive corruption model, the breakdown point is $1/3$ for high-dimensional halfspace-symmetric distributions. We show that under $\TV$ corruptions, the breakdown point reduces to $1/4$ for the same set of distributions. We also show that a certain projection algorithm can attain the optimal breakdown point of $1/2$.
  Both the Tukey median estimator and the projection algorithm achieve sample complexity linear in dimension. 
\end{abstract}

\tableofcontents
\newpage

\section{Introduction}





The Tukey median is the point(s) with largest Tukey depth~\citep{tukey1975mathematics}; it is a generalization of the one-dimensional median to high dimensions (see \eqref{eq.def_Tukey} 
for a formal definition). 
%
%
Its behavior is well-understood under 
the additive, or Huber, corruption model~\citep{huber1973robust} in which 
an $\epsilon$-fraction of the data are 
arbitrary outliers.
It is first shown in~\citet{donoho1982breakdown, donoho1992breakdown} that the breakdown point for Tukey median is 
$1/3$ for halfspace-symmetric distributions in dimension $d\geq 2$, and the breakdown point is $1/(d+1)$ without the halfspace-symmetric assumption. Further analyses in~\citet{chen2002influence} quantify the influence function and maximum bias for  halfspace-symmetric distributions, and the finite-sample behavior for elliptical distributions is analyzed in~\citet{chen2018robust}. 

In this paper, we consider the stronger $\TV$ corruption model, which allows both adding and deleting mass from the original distribution. 
We quantify the maximum bias of Tukey median and provide both upper and lower bounds for the breakdown point under $\TV$ corruptions. 
Interestingly, the breakdown point for halfspace-symmetric distributions in high dimensions decreases from $1/3$ under additive corruptions to 
$1/4$ under $\TV$ corruptions. We show that a different algorithm,  projection under the halfspace metric, has breakdown point $1/2$ in the same setting, which is the maximum breakdown point any translation-equivariant estimator can achieve~\cite[Equation 1.38]{rousseeuw2005robust}. We summarize the breakdown point for different algorithms in Figure~\ref{fig:breakdown}.

We extend the population results on maximum bias and breakdown point under $\TV$ corruptions 
to the finite-sample case, showing that we approach the infinite-data limit within a 
constant factor once the number of samples $n$ is linear in $d$. Our analysis holds under 
both the oblivious and adaptive models
considered in the literature~\citep{zhu2019generalized}. 

\begin{figure}[!htbp]
    \centering
    \includegraphics[width=0.7\linewidth]{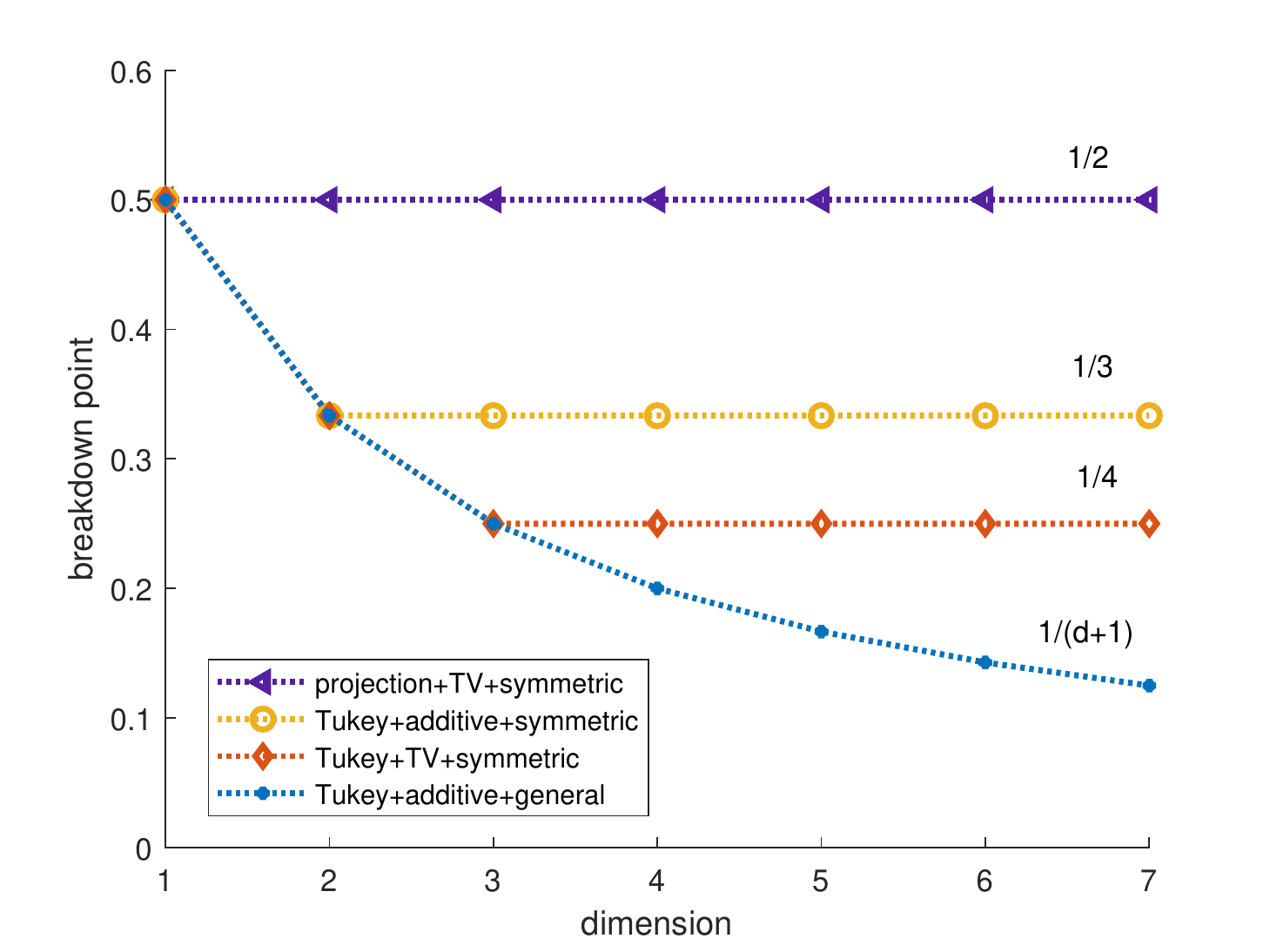}
    \caption{Summary of the breakdown point of different algorithms. Here `Tukey' denotes Tukey median and `projection' denotes the projection algorithm.  `Additive' and `TV' are the two corruption models. `Symmetric' denotes the family of halfspace-symmetric distribution, and `general' denotes the family of all distributions.} 
    \label{fig:breakdown}
\end{figure}

\section{Preliminaries}
We provide definitions for the Tukey median,  halfspace-symmetric distributions,   the additive and $\TV$ corruption models, maximum bias, and breakdown point. 
\subsection{Tukey median}


For any distribution $p$ and $\mu \in \bR^d$, the Tukey depth is defined as the minimum probability density on one side of a hyperplane through $\mu$: 
\begin{align}
\label{eq.def_Tukey}
    D_{\mathsf{\mathsf{Tukey}}}(\mu, p) = \inf_{v \in \bR^d} p(v^\top(X-\mu)\geq 0).
\end{align}
The Tukey median of a distribution $p$ is defined as the point(s) with largest Tukey depth: 
\begin{align}
    T(p) = \argmax_{\mu \in\bR^d} D_{\mathsf{\mathsf{Tukey}}}(\mu, p).
\end{align}
When $d=1$, Tukey median reduces to median. The Tukey median may not be unique even in one dimension. When the maximizer for Tukey depth is not unique, we use $T(p)$ to denote the set for all the maximizers and refer to this set as the Tukey median for distribution $p$.

\subsection{Halfspace-symmetric distributions}
We adopt the definition of halfspace-symmetric distributions from~\citep{zuo2000general, chen2002influence}. We say a distribution $p$ is halfspace-symmetric if there exists a point $\mu \in\bR^d$ such that for $X \sim p$,  $(X-\mu)$ and $-(X-\mu)$ are equal in distribution for all univariate projections, i.e. 
\begin{align}
    \forall v \in\bR^d, v^\top (X-\mu) \buildrel d \over = -v^\top (X-\mu).
\end{align}
Here $\buildrel d \over =$ represents equal in distribution. We call the point $\mu$ the center of the distribution $p$.
The
class of halfspace-symmetric distributions contains both the class of centro-symmetric distributions in~\citet{donoho1992breakdown} and elliptical distributions in~\citet{chen2018robust}. 
For a halfspace-symmetric distribution $p$, $\mu$ is the mean of $p$. The Tukey depth satisfies
$D_{\mathsf{Tukey}}(\mu, p)\geq 1/2$ and the Tukey median $T(p)$ contains $\mu$. 

\subsection{Population corruption models}
In the population level, 
we consider two corruption models: additive corruption model and $\TV$ corruption model~\cite{donoho1988automatic, diakonikolas2017being, zhu2019generalized}. 

\paragraph{Additive corruption model}
In a level-$\epsilon$ additive corruption model, given some true distribution $p^*$,  the adversary can generate  corrupted distribution $p = (1-\epsilon)p^* + \epsilon r$, where $\epsilon \in[0, 1)$ is the level of corruption, and $r\in\mathbb{M}^d$ is an arbitrary distribution selected by adversary.  We denote the set for all possible $\epsilon$-additive corruptions from $p^*$ as 
\begin{align}
    \mathcal{C}_{\mathsf{add}}(p^*, \epsilon) = \{ (1-\epsilon)p^* + \epsilon r \mid r \in \mathbb{M}^d\}.
\end{align}

\paragraph{Total variation distance corruption model}
The total variation distance between  two distributions $p, q$ is defined as
\begin{align}
    \TV(p, q) = \sup_{A} p(A) - q(A).
\end{align}

In a level-$\epsilon$ $\TV$ corruption model, 
given some true distribution $p^*$, the adversary can generate any corrupted distribution $p$ with $\TV(p, p^*)\leq \epsilon$. For any $p\in  \mathcal{C}_{\mathsf{add}}(p^*, \epsilon)$, it is always true that $\TV(p^*, p)= \sup_{A} \epsilon (p^*(A)-r(A))\leq \epsilon$. Thus the $\TV$ corruption model is a stronger corruption model than the additive corruptions, since $\TV$ corruptions  allow not only additive corruption, but also deletion and replacement.

\subsection{Maximum bias and breakdown point for Tukey median}

Given a fixed distribution $p^*$, 
the maximum bias $b(p^*, \epsilon)$ for Tukey median is defined as the maximum distance 
between $T(p)$  and $T(p^*)$, where $p$ is in the set of all possible level-$\epsilon$ corruptions: 
\begin{align}
    b_{\mathsf{add}}( p^*, \epsilon)& =  \sup_{p\in \mathcal{C}_{\mathsf{add}}(p^*, \epsilon), x\in T(p), y\in T(p^*)} \|x-y\|, \\
    b_{\TV}(p^*, \epsilon)& =  \sup_{\TV(p^*, p)\leq \epsilon, x\in T(p), y\in T(p^*)} \|x-y\|.
\end{align}
The corresponding breakdown point $\epsilon^*(p^*)$ is defined as the minimum corruption level that can drive the maximum bias to infinity:
\begin{align}
    \epsilon^*_{\mathsf{add}}(p^*) & = \inf \{\epsilon \mid  b(p^*, \epsilon) = \infty\},\\
    \epsilon^*_{\TV}(p^*) & = \inf \{\epsilon \mid b(p^*, \epsilon) = \infty\}.
\end{align}
Based on the definition of the breakdown point for a single distribution, 
we define the breakdown point for a family of distribution $\GG$ as the worst breakdown point for any distribution inside $\GG$, i.e.
\begin{align}
    \epsilon^*_{\mathsf{add}}(\GG) = \inf_{q \in \GG} \epsilon_{\mathsf{add}}^*(q), \quad  \epsilon^*_{\TV}(\GG) = \inf_{q \in \GG} \epsilon_{\TV}^*(q).
\end{align}



\section{Population analysis of Tukey median}
In this section, we quantify the maximum bias and the breakdown point of the Tukey median in population level. 
The maximum bias of the Tukey median for halfspace-symmetric distributions under additive corruption model is determined in \cite[Theorem 3.4]{chen2002influence}, which shows that the worst-case perturbation is to add a single point with mass $\epsilon$.
It is also shown in \cite[Theorem 2.1]{chen2018robust} that under additive corruptions, the Tukey median achieves near optimal maximum bias for
mean estimation if the true distribution $p^*$ belongs to the family of elliptical distributions.

Here we demonstrate a gap in the breakdown point for halfspace-symmetric distributions between the additive and $\TV$ corruption models.  
\begin{theorem}\label{thm.breakdown}
Denote $\GG_{\mathsf{half}}$ as the set of all halfspace-symmetric distributions. Then the breakdown point for $\GG_{\mathsf{half}}$ is
\begin{align*}
    \epsilon^*_{\mathsf{add}}(\GG_{\mathsf{half}}) = \begin{cases} 1/2, & d = 1 \\ 1/3, & d\geq 2
    \end{cases}, \quad
    \epsilon^*_{\TV}(\GG_{\mathsf{half}}) = \begin{cases} 1/2, & d = 1 \\ 1/3, & d = 2 \\ 1/4, & d \geq 3
    \end{cases}
\end{align*}

\end{theorem}
\begin{proof}[Proof of Theorem~\ref{thm.breakdown}]
We first show the upper bound for both breakdown points. We defer the lower bound to Theorem~\ref{thm.population_tukey}. The construction of upper bound is summarized in Figure~\ref{fig:breakdown}.

\begin{figure}[!htbp]
    \centering
    \includegraphics[width=0.95\linewidth]{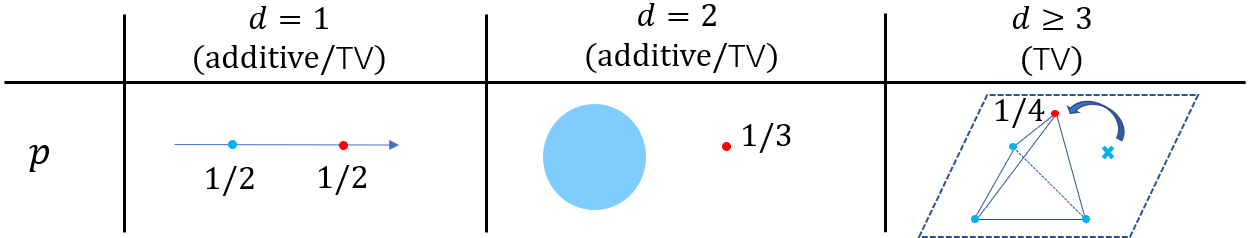}
    \caption{Illustration of worst case distributions achieving the breakdown point. Blue represents the original probability mass in $p^*$, blue cross represents deleted points and red represents added points by adversary. In all three cases, the red point is a Tukey median of $p$. Thus by driving the red point to infinity the estimator also goes to infinity. }
    \label{fig:tukey}
\end{figure}

For $d=1$, by adding $1/2$ mass onto $z$ and letting $z\rightarrow +\infty$, the maximum bias can be driven to infinity. Thus $\epsilon^*(\GG_{\mathsf{half}}) \leq 1/2$ under both corruption models. 

For $d\geq 2$  under additive corruption model, 
the upper bound of breakdown point $\epsilon^*_{\mathsf{add}}$ is proven in~\cite[Proposition 3.3]{donoho1992breakdown}. For completeness we sketch the proof here. Consider $p^*$ as a uniform distribution supported on unit ball. The adversary adds $1/3$ probability mass onto a point $\mu\in\bR^d$ outside the unit ball to get a new distribution $p$. Then $D_{\mathsf{Tukey}}(\mu, p) = 1/3$. On the other hand, for any point $\mu'\neq \mu$, if $\mu'$ is outside unit ball, there must exist a hyperplane which goes through $\mu'$ such that the unit ball is on one side of the hyperplane. Thus   $D_{\mathsf{Tukey}}(\mu', p)\leq 1/3$. If $\mu'$ is inside unit ball, consider any hyperplane that goes through $0$ and $\mu'$. The mass of the side of hyperplane which does not contain $\mu$ is $1/3$\footnote{If $\mu$ is on the same hyperplane. One can slightly rotate the hyperplane such that $\mu$ is not on it. This still guarantees the corresponding depth to be arbitrarily close to $1/3$.}. Thus  we also have $D_{\mathsf{Tukey}}(\mu', p) \leq 1/3$. Overall  $\mu$ must be one of the Tukey median for $p$. By setting $\mu \rightarrow \infty$ the proof is done. 

Since $\TV$ corruption model is a stronger corruption model, the upper bound of breakdown points for $d=1,2$ under $\TV$ corruptions readily follows from that under additive corruptions. Now we show the upper bound for $\epsilon^*_{\TV}$ when $d\geq 3$. 

Consider the following example as illustrated in Figure~\ref{fig:tukey}: in a 3-dimensional space, $p^*$ is a distribution with equal probability on the four nodes of a 2-dimensional square. To be precise, $p^*(X = t) = 1/4$ for any $t \in \{(-1, -1, 0), (-1, 1, 0),(1, -1, 0),(1, 1, 0)\}$. Thus $p^*$ is a halfspace-symmetric distribution, and $T(p^*) = (0, 0, 0)$ gives a unique Tukey median for $p^*$. 

Now we move one of the point $(1, 1, 0)$ to $(-0.5, -0.5, z)$ to get corrupted distribution $p$, where $z >0$. Now the four points form a tetrahedron. For any point $\mu$ that is inside the tetrahedron, the Tukey depth $D_{\mathsf{Tukey}}$ is always $1/4$. For any point that is outside the tetrahedron, the Tukey depth is always $0$. Thus all the points inside the tetrahedron are a Tukey median for the corrupted distribution $p$. By taking $z \rightarrow +\infty$, the Tukey median $T(p)$ is driven to infinity.
Thus we know that $\epsilon_{\TV}^*(\GG_{\mathsf{half}})\leq 1/4$ when $d\geq 3$. 

\end{proof}

Without the halfspace-symmetric assumption, the breakdown point for Tukey median under additive corruption model is $1/(d+1)$,  which is shown in~\cite[Proposition 2.3]{donoho1992breakdown}. This is also true under $\TV$ corruption model. 

To better illustrate the behavior of Tukey median, we analyze the maximum bias for halfspace-symmetric distributions. The performance guarantee relies on the decay function of the distribution, which characterizes how much probability mass is around the center of distribution. Assume the true distribution $p^*$ is halfspace-symmetric centered at $\mu^*$. 
We define the decay function $h(t): \mathbb{R}_{\geq 0} \mapsto \mathbb{R}_{\geq 0}$ as 
\begin{align}\label{eqn.def_h}
    h(t) = \sup_{v\in\bR^d, \|v\|_*\leq 1}  p^*(v^\top (X - \mu^*) > t),
\end{align}
where $\|\cdot\|_*$ is the dual norm of $\|\cdot \|$. 
Note that $h$ is a non-increasing non-negative function and $h(0) = 1 - D_{\mathsf{Tukey}}(\mu^*, p^*) \leq 1/2$ for halfspace-symmetric distributions. 

In the next theorem, we show that the maximum bias is controlled if the distribution has enough mass around its center:

\begin{theorem}\label{thm.population_tukey}
Assume $p^*$ is halfspace-symmetric with center $\mu^*$ and decay function $h(t)$ defined in~\eqref{eqn.def_h}. Then the maximum bias satisfies:
\begin{align}\label{eqn.add_maximum_bias}
    b_{\mathsf{add}}(p^*, \epsilon)& \leq \begin{cases}
     h^{-1}\left (\max(\frac{(1-\epsilon)(1-h(0))-\epsilon}{(1-\epsilon)}, \frac{1/2-\epsilon}{1-\epsilon})\right ), & d = 1\\
     h^{-1}\left (\max(\frac{(1-\epsilon)(1-h(0))-\epsilon}{(1-\epsilon)}, \frac{1/3-\epsilon}{1-\epsilon})\right ), & d = 2 \\
     h^{-1}\left (\frac{(1-\epsilon)(1-h(0))-\epsilon}{(1-\epsilon)}\right ), & d\geq 3
    \end{cases}
   , \\ 
    b_{\TV}(p^*, \epsilon)&  \leq \begin{cases}
     h^{-1}\left (\max(1 - h(0)-2 \epsilon, 1/2-\epsilon)\right ), & d = 1\\
     h^{-1}\left (\max(1 - h(0)-2 \epsilon, 1/3-\epsilon)\right ), & d=2 \\
     h^{-1}\left (1 - h(0)-2 \epsilon\right ), & d\geq 3 \label{eqn.TV_maximum_bias}
    \end{cases}.
\end{align}
Here  $h^{-1}$ is the generalized inverse function of $h$ defined as
\begin{align}\label{eqn.def_inverse}
    h^{-1}(y) = 
    \inf\{x \mid h(x)< y\}.
\end{align}
\end{theorem}

For Gaussian distribution with the operator norm of covariance bounded,  $\mu^*$ is the mean and $h(t) = 1/2 - \Theta(t)$ for $t$ small, and Theorem~\ref{thm.population_tukey} implies that Tukey median achieves the  maximum bias $O(\epsilon)$ for robust Gaussian mean estimation, which is known to be optimal up to constant factor. 

For any fixed distribution $p^*$, it suffices to have $t>0$ for $h^{-1}(t)$ to be finite.
Thus as a direct corollary of Theorem~\ref{thm.population_tukey}, it provides tight lower bound on the breakdown point of halfspace symmetric distributions in Theorem~\ref{thm.breakdown} via noting that $h(0) \leq 1/2$ for halfspace-symmetric distributions.   

The results in Theorem~\ref{thm.population_tukey} can be extended beyond halfspace symmetric distributions. For any true distribution $p^*$, since the Tukey median of $p^*$ may not be unique, we define the new $h(t)$ as
   \begin{align}
        h(t) = \sup_{v\in\bR^d, \|v\|_*\leq 1, \mu^*\in T(p^*)}  p^*(v^\top (X - \mu^*) > t).
   \end{align}
   Then following the same argument as the proof, the result in~\eqref{eqn.add_maximum_bias} and~\eqref{eqn.TV_maximum_bias} still hold.

\begin{proof}[Proof of Theorem \ref{thm.population_tukey}]
We first show that it suffices to bound $D_{\mathsf{Tukey}}(T(p), p^*)$ via the following lemma:
\begin{lemma}\label{lem.lower_bound}
Under the same condition as Theorem~\ref{thm.population_tukey}, if $D_{\mathsf{Tukey}}(T(p), p^*)\geq \alpha$, we have
\begin{align}
    \|T(p) - \mu^*\| \leq h^{-1}(\alpha).
\end{align}
\end{lemma}
\begin{proof}[Proof of Lemma~\ref{lem.lower_bound}]
Let $\tilde{v} = \argmax_{\|v\|_* \leq 1} v^\top (T(p) - \mu^*)$. 
Indeed, for any $t$ such that $h(t) < \alpha$, if $\| T(p) -\mu^*\| >t$, we have
\begin{align}
 D_{\mathsf{\mathsf{Tukey}}}(T(p), p^*) 
\leq &{p^*}(\tilde{v}^\top(X - T(p)) \geq 0) \nonumber \\
= & {p^*}(\tilde{v}^\top(X - \mu^*) \geq  \| T(p) -\mu^*\|) \nonumber \\
\leq &  {p^*}(\tilde{v}^\top(X - \mu^*) >t ) \nonumber \\
 \leq   &  h(t) 
 <     \alpha,
\end{align}
resulting in a contradiction. Thus the lemma holds. 
\end{proof}
Now it suffices to lower bound $D_{\mathsf{Tukey}}(T(p), p^*)$ for different dimensions and different corruption models. 

Under the $\TV$ corruption model, from the definition of $D_{\mathsf{Tukey}}$ and $\TV$, we have for any $\mu\in\bR^d$,
\begin{align}
    & D_{\mathsf{\mathsf{Tukey}}}(\mu, p) - D_{\mathsf{\mathsf{Tukey}}}(\mu, p^*) \nonumber \\  
    = & \inf_{v\in\bR^d} p(v^\top (X - \mu)\geq 0) - \inf_{v\in\bR^d} p^*(v^\top (X - \mu)\geq 0) \nonumber \\
     \leq & \sup_{v\in\bR^d} p^*(v^\top (X - \mu)< 0) - p(v^\top (X - \mu)< 0) \nonumber \\ 
     \leq &  \TV(p, p^*)\leq \epsilon.
\end{align}
Under the additive corruption model, we have a tigher bound:  from the definition we know that $p(A) = (1-\epsilon)p^*(A) +\epsilon r(A)\leq  (1-\epsilon)p^*(A) +\epsilon$ for any event $A$. Thus 
\begin{align}
    D_{\mathsf{\mathsf{Tukey}}}(\mu, p) & = \inf_{v\in\bR^d} p(v^\top (X - \mu)\geq 0) \nonumber \\
    & \leq \inf_{v\in\bR^d} (1-\epsilon)p^*(v^\top (X - \mu)\geq 0) + \epsilon \nonumber \\
    & = (1-\epsilon) D_{\mathsf{\mathsf{Tukey}}}(\mu, p^*) + \epsilon.
\end{align}

When $d=1$, we know that $  D_{\mathsf{\mathsf{Tukey}}}(T(p), p)\geq 1/2 $ for any distribution $p$. Thus $ D_{\mathsf{\mathsf{Tukey}}}(T(p), p^*) \geq  \frac{D_{\mathsf{\mathsf{Tukey}}}(T(p), p) - \epsilon}{1-\epsilon}\geq \frac{1/2-\epsilon}{1-\epsilon}$ under additive corruption, $ D_{\mathsf{\mathsf{Tukey}}}(T(p), p^*) \geq  D_{\mathsf{\mathsf{Tukey}}}(T(p), p) - \epsilon\geq 1/2-\epsilon$ under $\TV$ corruption.

When $d=2$, we know that  $  D_{\mathsf{\mathsf{Tukey}}}(T(p), p)\geq 1/3 $ for any distribution $p$ from~\cite[Proposition 2.3]{donoho1992breakdown}.
Thus following the same argument as $d=1$, we have  $ D_{\mathsf{\mathsf{Tukey}}}(T(p), p^*) \geq   (1/3-\epsilon)/(1-\epsilon)$ under additive corruption, $ D_{\mathsf{\mathsf{Tukey}}}(T(p), p^*) \geq 1/3-\epsilon$ under $\TV$ corruption.

For arbitrary dimension under additive corruption model, we also have another lower bound:
\begin{align}
    D_{\mathsf{\mathsf{Tukey}}}(T(p), p^*)& \geq  (D_{\mathsf{\mathsf{Tukey}}}(T(p), p) - \epsilon)/(1-\epsilon) \nonumber \\
    & \geq (D_{\mathsf{\mathsf{Tukey}}}(\mu^*, p) - \epsilon)/(1-\epsilon) \nonumber \\ 
    & \geq ((1-\epsilon)D_{\mathsf{\mathsf{Tukey}}}(\mu^*, p^*) - \epsilon)/(1-\epsilon)\nonumber \\
    & = {((1-\epsilon)(1-h(0))-\epsilon)}/{(1-\epsilon)}.
\end{align}
Here we use that $p(A) = (1-\epsilon)p^*(A) +\epsilon r(A)\geq  (1-\epsilon)p^*(A) $ for any event $A$. For $\TV$ corruption model, we have
\begin{align}
    D_{\mathsf{\mathsf{Tukey}}}(T(p), p^*)& \geq  D_{\mathsf{\mathsf{Tukey}}}(T(p), p) - \epsilon  \geq D_{\mathsf{\mathsf{Tukey}}}(\mu^*, p) - \epsilon \nonumber \\ 
    & \geq D_{\mathsf{\mathsf{Tukey}}}(\mu^*, p^*) - 2\epsilon = 1 - h(0) - 2\epsilon. \nonumber 
\end{align}
Combining the lower bounds with Lemma~\ref{lem.lower_bound} gives the proof.
\end{proof}

\section{Finite sample analysis of Tukey median}

In this section, we extend the  population results in the previous section to finite-sample case. 

Given finite samples, there are two different corruption models: oblivious corruption and adaptive corruption~\citep{zhu2019generalized}. In the oblivious corruption, the adversary first picks a corrupted population distribution $p$ from $C_{\TV}(p^*, \epsilon)$, then we take $n$ samples from $p$. In the adaptive corruption, we first take $n$ samples from $p^*$, then the adversary samples $n'$ from some distribution that is stochastically dominated 
by a binomial distribution $n'\sim \mathsf{B}(n, \epsilon)$ and replace $n'$  points in the samples by arbitrary points to get the corrupted empirical distribution $\hat p_n$. It is shown in~\citet{diakonikolas2019robust, zhu2019generalized} 
that adaptive corruption model is a stronger  corruption model than oblivious corruptions.  

Now we bound the maximum bias in the finite-sample case. We show that with $d/\epsilon^2$ samples, the estimation error can be of the same order as the population error in Theorem~\ref{thm.population_tukey}: 


\begin{theorem}\label{thm.finite_sample_tukey}
Assume the true distribution $p^*$ is halfspace-symmetric centered at $\mu^*$ with decay function $h(t)$ defined in~\eqref{eqn.def_h}.
Denote $\hat p_n$ as the corrupted empirical distribution under either  oblivious  or adaptive $\TV$ corruptions  of level $\epsilon$. 
When $d\geq 3$,   with probability at least $1-\delta$, there exists universal constant $C>0$ such that for any $\hat \mu \in T(\hat p_n)$ as the Tukey median of $\hat p_n$, 
\begin{align}
    \| \hat \mu - \mu^*\| \leq h^{-1}\left (1 - h(0)-2\tilde \epsilon\right )
\end{align}
when $2\tilde \epsilon< 1-h(0)$. Here $\tilde \epsilon = \epsilon + C\cdot \sqrt{\frac{d+1+\log(1/\delta)}{n}}$ 
,  $h^{-1}$ is the generalized inverse function of $h$ defined in~\eqref{eqn.def_inverse}.
\end{theorem}

\begin{proof}
It suffices to show the result for adaptive corruption model. 
From Lemma~\ref{lem.lower_bound}, we know that it also suffices to lower bound $D_{\mathsf{Tukey}}(\hat \mu, p^*)$, where $\hat \mu\in T(\hat p_n)$.  

We introduce the halfspace metric defined in~\citet{donoho1988automatic} as 
\begin{align}\label{eqn.tTV_def}
    \tTV(p, q) = \sup_{v\in\bR^d, t\in\bR} | p(v^\top X \geq t) - q(v^\top X \geq t)|.
\end{align}
From the definition we have 
$\tTV(p, q)\leq \TV(p, q)$ for all $p,q$. 
We first show that 
$|D_{\mathsf{\mathsf{Tukey}}}(\mu,p) - D_{\mathsf{\mathsf{Tukey}}}(\mu,q)| \leq \tTV(p,q)$
for any two distributions $p, q$ and any $\mu\in\bR^d$. 
To see this, note that the left hand side is
\begin{align}
    &|D_{\mathsf{\mathsf{Tukey}}}(\mu,p) - D_{\mathsf{\mathsf{Tukey}}}(\mu,q)| \nonumber \\
     = & \inf_{v\in\bR^d} p(v^\top (X - \mu)\geq 0) - \inf_{v\in\bR^d} q(v^\top (X - \mu)\geq 0) \nonumber \\
     \leq & \sup_{v\in\bR^d} q(v^\top (X - \mu)< 0) - p(v^\top (X - \mu)< 0) \leq \tTV(p, q).\nonumber
\end{align}
For Tukey median $\hat \mu = T(\hat p_n) = \argmax_{\mu\in\bR^d} D_{\mathsf{\mathsf{Tukey}}}(\mu, \hat p_n)$, 
\begin{align}
    D_{\mathsf{\mathsf{Tukey}}}(\hat \mu, p^*)& \geq  D_{\mathsf{\mathsf{Tukey}}}(\hat \mu, \hat p_n) - \tTV(\hat p_n, p^*) \nonumber \\
    & \geq D_{\mathsf{\mathsf{Tukey}}}(\mu^*, \hat p_n) - \tTV(\hat p_n, p^*) \nonumber \\ 
    & \geq D_{\mathsf{\mathsf{Tukey}}}(\mu^*, p^*) - 2\tTV(\hat p_n, p^*). \nonumber
\end{align}
Now let $\hat{p}_n^*$ be the uncorrupted 
distribution, so that $\hat{p}_n$ is obtained 
from $\hat{p}_n^*$ by modifying part of samples as in adaptive corruption model. Then by triangle inequality of $\tTV$, 
\begin{align}
      D_{\mathsf{\mathsf{Tukey}}}(\hat \mu, p^*)
    & \geq D_{\mathsf{\mathsf{Tukey}}}(\mu^*, p^*) - 2\tTV(\hat p_n, \hat p_n^*) -  2\tTV(\hat p_n^*, p^*) \nonumber  \\
    & \geq 1 - h(0)-2\TV(\hat p_n, \hat p_n^*) -  2\tTV(\hat p_n^*, p^*).\nonumber
\end{align}
where we repeatedly use the fact that for any $p,q,\mu$, we have $|D_{\mathsf{\mathsf{Tukey}}}(\mu,p) - D_{\mathsf{\mathsf{Tukey}}}(\mu,q)| \leq \tTV(p,q)$. Here $\hat p_n \mid \hat{p}_n^*$ follows adaptive corruption model. 
Now we upper bound the two terms $\TV(\hat p_n, \hat p_n^*)$ and $ \tTV(\hat p_n^*, p^*)$. 
From \cite[Lemma B.1]{zhu2019generalized}, we know that 
 with probability at least $1-\delta$,
\begin{align}
    \TV(\hat p_n, \hat p^*_n)\leq (\sqrt{\epsilon} + \sqrt{\frac{\log(1/\delta)}{2n}})^2.
\end{align}

For the second term $\tTV(\hat p_n^*, p^*)$, from the VC inequality~\cite[Chap 2, Chapter 4.3]{devroye2012combinatorial} and the fact that the family of sets $\{\{x \mid v^\top x \geq t\} \mid \|v\|=1, t\in \bR, v\in \bR^d\}$ has VC dimension $d+1$, there exists  some universal constant $C^{\mathsf{vc}}$ such that with probability at least $1-\delta$:
\begin{align}\label{eqn.vc_inequality}
    \tTV(p^*,\hat{p}^*_n) &\leq C^{\mathsf{vc}} \cdot \sqrt{\frac{d+1+\log(1/\delta)}{n} }.
\end{align}

Denote $\tilde \epsilon = (\sqrt{\epsilon} + \sqrt{\frac{\log(1/\delta)}{2n}})^2+C^{\mathsf{vc}} \cdot \sqrt{\frac{d+1+\log(1/\delta)}{n} }$. 
Combining the two lemmata together, we know that with probability at least $1-2\delta$, 
$  D_{\mathsf{\mathsf{Tukey}}}(\hat \mu, p^*) \geq  1 - h(0) - 2\tilde \epsilon.$
The proof is completed by combining the result with Lemma~\ref{lem.lower_bound}. 
\end{proof}
As a direct corollary of the finite sample result, we can show that for Gaussian distribution the estimation error is $O(\epsilon)$ with sample complexity $O(d/\epsilon^2)$. 
We remark that with the same proof, the population results in Theorem~\ref{thm.population_tukey} for $d=1, 2$ and additive corruptions can all be extended to finite-sample results with sample complexity $O(d/\epsilon^2)$. 
Similarly the halfspace-symmetric assumption can be discarded.

\section{$\tTV$ Projection Algorithm}


In the previous two sections, 
we show that Tukey median can achieve breakdown point $1/4$ for halfspace symmetric distributions under $\TV$ corruptions and the sample complexity is linear in dimension. In this section, we show that projection under halfspace metric $\tTV$, as defined in~\eqref{eqn.tTV_def}, is able to improve the breakdown point to $1/2$ under the same conditions. The $\tTV$ projection algorithm is first proposed in~\citet{donoho1988automatic} for robust mean estimation, and later generalized in~\citet{zhu2019generalized} for general robust inference problems.  
 
Denote $\GG(h)$ as the set of halfspace-symmetric distributions with controlled cumulative density function around its center:
\begin{align}
    \GG(h) = \{p \mid & X\sim p \text{ is halfspace-symmetric around } \mu \text{ and } \nonumber \\
     & \sup_{v\in\bR^d, \|v\|_*\leq 1}  p(v^\top (X - \mu) > t) \leq h(t)\}.
\end{align}
The $\tTV$ projection algorithm 
projects the corrupted empirical distribution $p$ onto the set $\GG(h)$  under $\tTV$ distance, i.e.~the output is
\begin{align}\label{eqn.algorithm_tTV_proj}
    \hat \mu(p) = \bE_q[X], \text{ where } q  = \argmin_{q \in \GG(h)} \tTV(q, p).
\end{align}
Note that the $\tTV$ projection algorithm requires the knowledge of the set $\GG(h)$, while the Tukey median is agnostic to the distributional assumption on $p^*$. In return, the $\tTV$ projection algorithm achieves a breakdown point of $1/2$ and better maximum bias than the Tukey median, as shown 
in the following theorem:
\begin{theorem}\label{thm.tilde_TV}

Assume the true distribution $p^*$ is halfspace-symmetric centered at $\mu^*$ with decay function $h(t)$ defined in~\eqref{eqn.def_h}. 
Then for any $p$ with $\TV(p^*, p) \leq \epsilon$, the projection estimator  $\hat \mu(p)$ in~\eqref{eqn.algorithm_tTV_proj} satisfies 
\begin{align}
    \| \hat \mu(p) - \mu^*\| \leq 2h^{-1}\left (1/2 - \epsilon \right )
\end{align}
when $\epsilon<1/2$. 
Here $h^{-1}$ is the generalized inverse function of $h$ defined in~\eqref{eqn.def_inverse}.
\end{theorem}


\begin{proof}
By triangle inequality and the property of projection, 
\begin{align}
    \tTV(p^*, q)& \leq \tTV(p^*, p) + \tTV(p, q) \nonumber   \\
    & \leq \tTV(p^*, p) + \tTV(p, p^*) \nonumber  \\
    & =2\tTV(p^*, p) \leq 2\TV(p^*, p)\leq 2\epsilon.
\end{align}

We also know that $p^*, q \in \GG(h)$.   Let $\tilde{v} = \argmax_{\|v\|_* \leq 1} v^\top ( \hat \mu - \mu^*)$. We have
\begin{align}
    & q(\tilde{v}^\top (X - \frac{\mu^* + \hat \mu}{2}) < 0) \nonumber \\ 
     = & q(\tilde{v}^\top (X -  \hat \mu) < - \frac{\|\hat \mu - \mu^*\|}{2}) \leq  h(\frac{\|\mu^* - \hat \mu\|}{2}). 
\end{align}

We show that it implies for any $\epsilon<1/2$,
$\|\hat \mu - \mu^*\|  \leq 2 h^{-1} (1/2-\epsilon)$.

For any $t$ such that $h(t) < 1/2 -\epsilon $, if $\| \hat \mu -\mu^*\| > 2t$,
\begin{align}
& p^*(\tilde{v}^\top (X - \frac{\mu^* + \hat \mu}{2}) < 0)=1-{p^*}(\tilde{v}^\top(X - \mu^*)\geq  \frac{\| \hat \mu -\mu^*\|}{2}) \nonumber \\
\geq & 1-{p^*}(\tilde{v}^\top(X - \mu^*)>  t) 
 \geq     1 -  h(t)
>    1/2 + \epsilon.
\end{align}

On the other hand, from $\tTV(p^*, q)\leq 2 \epsilon $, we know that
\begin{align}
     p^*(\tilde{v}^\top (X - \frac{\mu^* + \hat \mu}{2}) < 0) 
    \leq & q(\tilde{v}^\top (X - \frac{\mu^*+\hat \mu}{2}) < 0)  +  2\epsilon \nonumber \\
    \leq &  h(\frac{\|\mu^* - \hat \mu\|}{2}) + 2 \epsilon
    <  1/2+\epsilon, \nonumber 
\end{align}
resulting in a contradiction.

\end{proof}
The population result can also be extended to finite-sample case by projecting $\hat p_n$ instead of $p$ under $\tTV$.  The proof follows the same technique in Theorem~\ref{thm.finite_sample_tukey}. The key to the success of $\tTV$ projection is that it allows us to check the halfspace that goes through the middle of $\mu^*$ and $\hat \mu$, while Tukey median is only allowed to check the halfspace that goes through $\mu^*$ and $\hat \mu$. Although projection under $\TV$ would also give the same population rate, the finite sample error can be huge since $\TV(\hat p_n, p)=1$.
For both Theorem~\ref{thm.finite_sample_tukey} and~\ref{thm.tilde_TV}, the results can be extended to a more general perturbation model of corruptions under $\tTV$ distance. 

\section{Open Problem}

Considering the $\TV$ corruption model, Tukey median is an affine-equivariant estimator with breakdown point $1/4$ in high dimensions and good finite sample error for halfspace-symmetric distributions. The $\tTV$ projection algorithm is not affine-equivariant, but achieves breakdown point $1/2$ and good finite sample error in the same set of distributions. Both algorithms may not be efficiently solvable. 

It is an open problem to find an estimator that is affine-equivariant, with breakdown point $1/2$ and good finite sample error for halfspace-symmetric distributions without considering computational efficiency.

\bibliographystyle{plainnat}
\bibliography{di}

\begin{thebibliography}{13}
\providecommand{\natexlab}[1]{#1}
\providecommand{\url}[1]{\texttt{#1}}
\expandafter\ifx\csname urlstyle\endcsname\relax
  \providecommand{\doi}[1]{doi: #1}\else
  \providecommand{\doi}{doi: \begingroup \urlstyle{rm}\Url}\fi

\bibitem[Chen et~al.(2018)Chen, Gao, and Ren]{chen2018robust}
Mengjie Chen, Chao Gao, and Zhao Ren.
\newblock Robust covariance and scatter matrix estimation under huber’s
  contamination model.
\newblock \emph{The Annals of Statistics}, 46\penalty0 (5):\penalty0
  1932--1960, 2018.

\bibitem[Chen et~al.(2002)Chen, Tyler, et~al.]{chen2002influence}
Zhiqiang Chen, David~E Tyler, et~al.
\newblock The influence function and maximum bias of tukey's median.
\newblock \emph{The Annals of Statistics}, 30\penalty0 (6):\penalty0
  1737--1759, 2002.

\bibitem[Devroye and Lugosi(2012)]{devroye2012combinatorial}
Luc Devroye and G{\'a}bor Lugosi.
\newblock \emph{Combinatorial methods in density estimation}.
\newblock Springer Science \& Business Media, 2012.

\bibitem[Diakonikolas et~al.(2017)Diakonikolas, Kamath, Kane, Li, Moitra, and
  Stewart]{diakonikolas2017being}
Ilias Diakonikolas, Gautam Kamath, Daniel~M Kane, Jerry Li, Ankur Moitra, and
  Alistair Stewart.
\newblock Being robust (in high dimensions) can be practical.
\newblock In \emph{Proceedings of the 34th International Conference on Machine
  Learning-Volume 70}, pages 999--1008. JMLR. org, 2017.

\bibitem[Diakonikolas et~al.(2019)Diakonikolas, Kamath, Kane, Li, Moitra, and
  Stewart]{diakonikolas2019robust}
Ilias Diakonikolas, Gautam Kamath, Daniel Kane, Jerry Li, Ankur Moitra, and
  Alistair Stewart.
\newblock Robust estimators in high-dimensions without the computational
  intractability.
\newblock \emph{SIAM Journal on Computing}, 48\penalty0 (2):\penalty0 742--864,
  2019.

\bibitem[Donoho(1982)]{donoho1982breakdown}
David~L Donoho.
\newblock Breakdown properties of multivariate location estimators.
\newblock Technical report, Technical report, Harvard University, Boston, 1982.

\bibitem[Donoho and Gasko(1992)]{donoho1992breakdown}
David~L Donoho and Miriam Gasko.
\newblock Breakdown properties of location estimates based on halfspace depth
  and projected outlyingness.
\newblock \emph{The Annals of Statistics}, 20\penalty0 (4):\penalty0
  1803--1827, 1992.

\bibitem[Donoho and Liu(1988)]{donoho1988automatic}
David~L Donoho and Richard~C Liu.
\newblock The ``automatic'' robustness of minimum distance functionals.
\newblock \emph{The Annals of Statistics}, 16\penalty0 (2):\penalty0 552--586,
  1988.

\bibitem[Huber(1973)]{huber1973robust}
Peter~J Huber.
\newblock Robust regression: asymptotics, conjectures and monte carlo.
\newblock \emph{The Annals of Statistics}, 1\penalty0 (5):\penalty0 799--821,
  1973.

\bibitem[Rousseeuw and Leroy(2005)]{rousseeuw2005robust}
Peter~J Rousseeuw and Annick~M Leroy.
\newblock \emph{Robust regression and outlier detection}, volume 589.
\newblock John wiley \& sons, 2005.

\bibitem[Tukey(1975)]{tukey1975mathematics}
John~W Tukey.
\newblock Mathematics and the picturing of data.
\newblock In \emph{Proceedings of the International Congress of Mathematicians,
  Vancouver, 1975}, volume~2, pages 523--531, 1975.

\bibitem[Zhu et~al.(2019)Zhu, Jiao, and Steinhardt]{zhu2019generalized}
Banghua Zhu, Jiantao Jiao, and Jacob Steinhardt.
\newblock Generalized resilience and robust statistics.
\newblock \emph{arXiv preprint arXiv:1909.08755}, 2019.

\bibitem[Zuo and Serfling(2000)]{zuo2000general}
Yijun Zuo and Robert Serfling.
\newblock General notions of statistical depth function.
\newblock \emph{Annals of statistics}, pages 461--482, 2000.

\end{thebibliography}

\end{document}